\documentclass{article}
\usepackage{arxiv}
\pdfoutput=1

\usepackage[utf8]{inputenc} 
\usepackage[T1]{fontenc}    
\usepackage{
  lipsum,
  natbib,
  amssymb,
  amsmath,
  amsthm,
  mathptmx,
  parskip,
  booktabs, 
  bm,
  bbm,
  graphicx,
  color,
  xcolor,
  tikz,
  caption,
  url,
  lineno,
  enumitem,
  hyperref,
  algorithm,
  algpseudocode,
  eucal
}

\graphicspath{ {./images/} }

\hypersetup{
  colorlinks = true,  
  urlcolor   = blue,  
  linkcolor  = black, 
  citecolor  = black  
}

\usetikzlibrary{math} 

\newcommand{\nei}[1]{\text{ne}(#1)}
\newcommand{\clo}[1]{\text{cl}(#1)}
\newcommand{\pa}[1]{\text{pa}(#1)}
\newcommand{\pas}[1]{\small\text{pa}_{#1}}

\newcommand{\amarg}[1]{\mathcal{I}_{#1}}
\newcommand{\df}{\textit{df}}
\newcommand{\dfh}{\textit{df}^{\hspace{1.5pt}h}}
\newcommand{\sse}{\textit{SSE}}

\DeclareRobustCommand\discdot{\tikz{\draw[fill = black,line width=0.5pt]  circle(0.6ex);}}
\DeclareRobustCommand\contdot{\tikz{\draw[fill = white,line width=0.5pt]  circle(0.6ex);}}

\newtheorem{proposition}[subsection]{Proposition}

\title{Detecting Outliers in High-dimensional Data with Mixed Variable Types using Conditional Gaussian Regression Models}

\author{
  Mads Lindskou \\
  Department of Mathematical Sciences\\
  Aalborg University\\
  \texttt{mlindsk@math.aau.dk} \\
  \And
  Torben Tvedebrink \\
  Department of Mathematical Sciences\\
  Aalborg University\\
  \texttt{tvede@math.aau.dk} \\
  \And
  \hspace{.7cm}Poul Svante Eriksen \\
  \hspace{.7cm}Department of Mathematical Sciences\\
  \hspace{.7cm}Aalborg University\\
  \hspace{.7cm}\texttt{svante@math.aau.dk} \\
  \And
  Niels Morling \\
  Section of Forensic Genetics, \\Department of
  Forensic Medicine, \\Faculty of Health and Medical Sciences,\\
  University of Copenhagen
  Aalborg University\\
  \texttt{niels.morling@sund.ku.dk}
}

\begin{document}
\maketitle
\begin{abstract}
Outlier detection has gained increasing interest in recent years, due to newly emerging technologies and the huge amount of high-dimensional data that are now available. Outlier detection can help practitioners to identify unwanted noise and/or locate interesting abnormal observations. To address this, we developed a novel method for outlier detection for use in, possibly high-dimensional, datasets with both discrete and continuous variables. We exploit the family of decomposable graphical models in order to model the relationship between the variables and use this to form an exact likelihood ratio test for an observation that is considered an outlier. We show that our method outperforms the state-of-the-art Isolation Forest algorithm on a real data example.
\end{abstract}


\section{Introduction}
\label{sec:introduction}

Outlier detection is an important learning paradigm and has drawn significant attention within the research community, as shown by the increasing number of publications in this field. An outlier in a data set is an observation that, for some reason, does not share the same characteristics as the majority (there may be more than one outlier) of all other observations. An outlier may be the most interesting observation in some situations. In other situations, it may be regarded as extreme noise, and it may be appropriate to remove it from the data. There is no clear mathematical definition of an outlier in the literature. \cite{hawkins1980identification} gave the following definition: ``an observation which deviates so much from the other observations in the data-set as to arouse suspicions that it was generated by a different mechanism''. In \cite{lindskououtlier}, this definition, was adapted by specifying a statistical hypothesis of an outlier being distributed differently than all other observations for discrete data sets. In this paper, we extend this definition to capture outliers in data sets with variables of mixed types, i.e.\ both discrete and continuous variables.

Most research on outlier detection has been focused on the two pure
cases where all variables are either discrete or continuous, while
little research has been done in the mixed case for high-dimensional
data sets \citep{garchery2018influence}. State-of-the art algorithms
include the Isolation Forest (iForest) algorithm
\citep{liu2008isolation} which has gained a lot of attention in recent
years. Many software packages implement iForest. In particular, the
procedure is implemented in the major data science languages, such as
R and Python, making it readily available to practitioners. Although
iForest was designed for outlier detection in the pure continuous
case, it is frequently used in the mixed case where the discrete
variables are transformed into continuous variables. However, the transformation may induce an unwanted ordering of or distance between the levels of the discrete variables. Hence, the information content of
the data may be altered by the transformation, but
nonetheless iForest is used in many papers on outlier detection 
with mixed data, and the performance is most often excellent, see
e.g.\ \cite{eiras2019large, xu2019mix, aryal2016revisiting,
  garchery2018influence} and \cite{aryal2019improved}. In the review
papers by \cite{domingues2018comparative} and \cite{emmott2015meta}, iForest is recommended as the best overall outlier detection
procedure and is recommended in production environments. In
Section~\ref{sec:pm}, we show that our method outperforms iForest as an
outlier detection method in the mixed case, when data includes a large
number of discrete variables. This is in contradiction the findings in the aforementioned papers.

We propose a novel probabilistic outlier detection method that relies
on the family of decomposable mixed graphical models, see e.g.\
\cite{lauritzen1996graphical}. Graphical models is a family of
statistical models, for which the dependencies between variables can
be depicted or read off from a graph called the interaction
graph. They are composed of a set of random variables (possibly both
continuous and discrete) and an interaction graph, for which each of
the vertices represents one of the random variables. In essence, a
graphical model encodes the conditional independencies between the
random variables. By imposing the graphs to be decomposable, the
likelihood function is ensured to be on closed form, which enables us
to express what is meant by an outlier using an exact likelihood ratio
test (LRT). Furthermore, the components of the proposed LRT are
composed of local information from the graph in terms of cliques
yielding a way to explore which variables have the largest impact of
the declaration of an observation as an outlier.

The likelihood function corresponds to a multivariate multiple regression model over the continuous variables, however due to local independencies, it is possible to restrict attention to simple multiple linear regressions with both continuous and discrete explanatory variables. In the pure discrete case, the method coincide with the one given in \cite{lindskououtlier} which will be apparent in Section \ref{sec:od}. In addition, the method also handles the pure continuous case which is shown to be equivalent to a sum of studentized residuals over local structures in the graph.

The LRT relies on simulating observations from the model, as in \cite{lindskououtlier}. However, in the mixed case we show that it is redundant to simulate the continuous counterpart of a simulated discrete observation. It turns out that the continuous contribution to the likelihood ratio can be drawn from a beta distribution once the discrete counterpart is known.
  

The rest of the paper is organised as follows. Section~\ref{sec:dmg} reviews the definitions of graphical models needed together with a useful proposition. Section~\ref{sec:ct} is devoted to the notation and likelihood function needed to arrive at the exact likelihood ratio test given in Section~\ref{sec:od}. Section~\ref{sec:outlier_test} summaries the proposed outlier detection method, ODMGM, in a pseudo algorithm providing a very detailed explanation of the steps to carry out. In Section~\ref{sec:real}, a real data set with mixed variable types is analysed, and the results are compared to those obtained with iForest.

\section{Decomposable Mixed Graphs}
\label{sec:dmg}

We focus on undirected graphs, i.e.\ graphical models for which the edges in the interaction graph are not directed. Such models are also known as Markov random networks. In the following, we introduce notation and concepts that can be found in e.g. \cite{lauritzen1996graphical}. An undirected mixed graph, $G = (V, E)$, is a pair consisting of a set of vertices $V$ and a set of edges $E = \{ \{u, v\} \mid u,v \in V, u \neq v \}$. Furthermore, $V$ is the union of $\Delta$, the discrete variables, and $\Gamma$, the continuous variables, where $|\Delta| = s, |\Gamma| = r$ and $m = |V| = s+r$. In the figures, we use circles to represent continuous variables and dots (discs) to represent discrete variables. 

A mixed graph is triangulated if it has no cycle of length $\geq 4$ without a chord. Mixed graphs are said to be decomposable if they are triangulated and do not contain any path between two non-adjacent discrete vertices passing through continuous variables only. Such paths are also called forbidden paths. For a decomposable mixed graph, $G$, with vertex set $V$, it holds that the sub-graph $G_{A} = (A, \{ \{ u,v\} \mid u,v \in A \, u \neq v\})$ is also decomposable for all subsets $A \subseteq V$. The subset $A$ is called a clique if $G_{A}$ is a complete graph, i.e. any two vertices are adjacent. Define the star graph of $G = (V,E)$ as
\[
  G^{\star} = (V^{\star}, E^{\star}) = \bigl(V \cup \{\star\}, E \cup \{ \{d, \star\} \mid d\in \Delta \}\bigr).
\] 
That is, $G^{\star}$ is the graph, in which $V$ is extended with the
$\star-$node, and all discrete vertices are connected to this. \cite{leimer1988triangulated} showed, that a graph is decomposable if
and only if the corresponding star graph is triangulated. It can,
therefore, be checked if a graph $G$ is decomposable by using the
maximum cardinality search (MCS) algorithm \citep{yannakakis1981computing} on $G^{\star}$.

Let $C_{1}, C_{2}, \ldots, C_{k}$ be a sequence of the cliques in an undirected graph, and define for $j = 1, 2 \ldots, k$
\begin{itemize}
\item $H_{j} = C_{1} \cup \cdots \cup C_{j}$,
\item $S_{j} = C_{j} \cap H_{j-1}$ and
\item $R_{j} = C_{j} \setminus H_{j-1}$,
\end{itemize}
where we define $H_{0}$ as the empty set. These sets are also referred
to as the histories, separators and residuals, respectively. The sequence is then said to be
perfect if the following conditions hold:
\begin{enumerate}[label={(\textsl{\alph*})}]
\item for all $j > 1$, there exist an index $i < j$ such that
  $S_{j} \subseteq C_{i}$,\label{enum:1}
\item all separators are complete, and
\item either $R_{j} \subseteq \Gamma$ or $S_{j} \subseteq \Delta$ for
  all $j > 1$.\label{enum:3}
\end{enumerate}
Condition \ref{enum:1} is known as the running intersection property, and
condition \ref{enum:3} ensures that no forbidden path exists. Denote by
$\nei{v} = \{ u \mid \{u, v\} \in E \}$ and
$\clo{v} = \nei{v} \cup \{v\}$ the neighbours and closure of the vertex
$v \in V$, respectively, and define
\begin{align}\label{eq:Bi}
B(v_{j}) = \clo{v_{j}} \cap \{ v_{1}, v_{2}, \ldots, v_{j} \}, \quad j > 1.
\end{align}
If the sets in \eqref{eq:Bi} form a perfect sequence of sets, the
sequence of vertices, $v_{1}, v_{2}, \ldots, v_{m}$, is said to be a
perfect numbering of the vertices. A graph is decomposable if and only
if, the vertices admit a perfect numbering and/or the cliques of the
graph can be perfectly numbered to form a perfect sequence
\citep{lauritzen1996graphical}. The cliques are then said to have the
running intersection property (RIP).


We use the notation $\pa{v} := B(v) \setminus\{v\}$ to denote the
parents of $v$ defined as the preceding numbered vertices of $v$ that
are also neighbours of $v$. The following result appears in
\citet[][p. 18]{lauritzen1996graphical} as a remark. However, this
result is vital for the model assumptions in Section~\ref{sec:ct}, and we therefore give a concise formal proof.

\begin{proposition}\label{prop:cont_before_disc}
  For decomposable mixed graphs, a perfect numbering of the vertices
  can be chosen such that the discrete variables are numbered before
  the continuous ones.
\end{proposition}

\begin{proof}
  Let $\delta$ be a discrete vertex numbered after the continuous
  vertex $\gamma$. By definition, $\gamma \notin \pa{\delta}$ and
  $\delta \notin \pa{\gamma}$. Hence, interchanging $\delta$ and
  $\gamma$ in the perfect sequence would leave $B(\gamma)$ and
  $B(\delta)$ unchanged and, thus, complete and satisfy
  $B(\delta) \subseteq \Delta$. After a suitable number of such
  interchanges, all continuous vertices will be preceded by the
  discrete vertices.
\end{proof}


Consider the mixed graph in Figure~\ref{fig:graph_example} (left), where $V = \{a,b,c,d,e,f\}$, and $$E = \bigl\{ \{a,b\}, \{b,c\}, \{b,d\}, \{c,d\}, \{c,e\}, \{d,e\} \bigr\},$$ with $\Delta = \{b, c, f\}$ and $\Gamma = \{a, d, e\}$. The corresponding star graph is depicted in Figure~\ref{fig:graph_example} (right), where it can be seen that the graph is triangulated and hence decomposable. A perfect numbering of the vertices is given as $c,b,f,d,e,a$.

\begin{figure}[h!]
  \centering
  \begin{minipage}{0.3\textwidth}
    \centering
    \begin{tikzpicture}
    \tikzmath{\eps = 0.6;}
    \tikzmath{\sep_ = .3cm;}
    \tikzset{dot/.style={thick, circle, draw = black, minimum size = \sep_, inner sep = 0, fill = black}}
    \tikzset{cir/.style={thick, circle, draw = black, minimum size = \sep_, inner sep = 0}}

    \node [dot, label = below:{$f$}] (f) at (-4*\eps, -\eps){};
    \node [cir, label = above:{$a$}] (a) at (-4*\eps,  \eps){};
    \node [dot, label = above:{$b$}] (b) at (-3*\eps,  \eps){};
    \node [dot, label = below:{$c$}] (c) at (-2*\eps, -\eps){};
    \node [cir, label = above:{$d$}] (d) at (-1*\eps,  \eps){};
    \node [cir, label = below:{$e$}] (e) at (0*\eps,  -\eps){};

    \draw [thick, draw=black] (a) -- (b);
    \draw [thick, draw=black] (b) -- (c);
    \draw [thick, draw=black] (b) -- (d);
    \draw [thick, draw=black] (c) -- (d);
    \draw [thick, draw=black] (c) -- (e);
    \draw [thick, draw=black] (d) -- (e);
  \end{tikzpicture}
\end{minipage}
\begin{minipage}{0.30\textwidth}
  \centering
  \begin{tikzpicture}
    \tikzmath{\eps = 0.6;}
    \tikzmath{\sep_ = .3cm;}
    \tikzset{dot/.style={thick, circle, draw = black, minimum size = \sep_, inner sep = 0, fill = black}}
    \tikzset{cir/.style={thick, circle, draw = black, minimum size = \sep_, inner sep = 0}}
    \tikzset{str/.style={thick, draw = white, minimum size = \sep_, inner sep = 0}}

    \node [dot, label = below:{$f$}] (f) at (-4*\eps, -\eps){};
    \node [cir, label = above:{$a$}] (a) at (-4*\eps, \eps){};
    \node [dot, label = above:{$b$}] (b) at (-3*\eps,  \eps){};
    \node [cir                     ] (s) at (-3*\eps,  -\eps){$\star$};
    \node [dot, label = below:{$c$}] (c) at (-2*\eps, -\eps){};
    \node [cir, label = above:{$d$}] (d) at (-1*\eps,  \eps){};
    \node [cir, label = below:{$e$}] (e) at (0*\eps,  -\eps){};

    \draw [thick, draw=black] (a) -- (b);
    \draw [thick, draw=black] (b) -- (s);
    \draw [thick, draw=black] (f) -- (s);
    \draw [thick, draw=black] (c) -- (s);
    \draw [thick, draw=black] (b) -- (c);
    \draw [thick, draw=black] (b) -- (d);
    \draw [thick, draw=black] (c) -- (d);
    \draw [thick, draw=black] (c) -- (e);
    \draw [thick, draw=black] (d) -- (e);
  \end{tikzpicture}
\end{minipage}
\caption{\label{fig:graph_example}A decomposable mixed graph (left) and
  its corresponding star graph (right), where \discdot \ represents
  discrete variables and \contdot \ represents continuous variables,
  respectively.}
\end{figure}
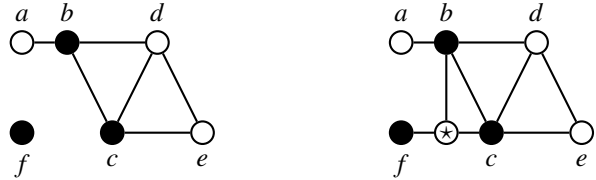

\section{Notation and the Likelihood Function}
\label{sec:ct}
Let $I$ be a $s-$dimensional discrete random vector and $Y$ a $r-$dimensional real random vector. A realised value of the random vector $X = (I, Y)$ is denoted $x = (i, y)$, where $i = (i_{d})_{d\in \Delta}$ is a tuple of discrete outcomes, also referred to as cell $i$, and $y = (y_{\gamma})_{\gamma\in \Gamma}$ is a real-valued vector. The state space of $I$ is denoted as $\mathcal{I} = \times_{d\in\Delta}\amarg{d}$, where $\amarg{d}$ is the level set of $d$; hence $i \in \amarg{}$. Marginal vectors are written $x_{a} = (i_{a \cap \Delta}, y_{a \cap \Gamma})$, where $i_{a \cap \Delta}$ and $y_{a\cap \Delta}$ are sub-vectors restricted to the sets $a \cap \Delta$ and $a \cap \Gamma$, respectively. We usually use the shorthand notations $i_{a} := i_{a \cap \Delta}$ and $y_{a} := y_{a\cap \Gamma}$. The level set for the $a-$marginal $i_{a}$ is denoted $\amarg{a}$; hence $i_{a} \in \amarg{a} = \times_{d\in a} \amarg{d}$. If the vertices are numbered, we write \[x = (x_{1}, x_{2}, \ldots, x_{m}) = (i_{1}, i_{2}, \ldots, i_{s}, y_{1}, y_{2}, \ldots, y_{r}).\] 

The observed counts in cell $i$ is denoted $n(i)$, and the probability of an observation falling in cell $i$ is $p(i)$. The $a-$marginal table is then defined by the counts
\[
  n_{a}(i_{a}) = \sum_{j: j_{a} = i_{a}} n(j),
\]
and similarly for the marginal probabilities $p_{a}$. For the empty set, $a = \emptyset$, we write $n(i_{\emptyset}) = |n|$, where $|n| = \sum_{i\in \amarg{}} n(i)$ is the total number of counts.

We assume that $I$ has probability mass function $p$, and that $Y$, given $I = i$, is a multivariate Gaussian model with mean and variance depending on cell $i$. A distribution of this form is called an inhomogeneous conditional Gaussian (CG) distribution \citep{lauritzen1989}. The joint density is written as
\begin{align}\label{eq:CG}
  f(x) = f(i,y) = f(y\mid i) \times p(i).
\end{align}
If the variance is assumed to be independent of the discrete
variables, the model is referred to as a homogeneous CG
distribution. The inhomogeneous case is treated below (with
some additional details in Appendix~\ref{sec:var_est}). The homogeneous case is discussed in Section~\ref{sec:hom}. Given a perfect numbering of the vertices, Proposition~\ref{prop:cont_before_disc} allows the
following factorisation of the joint density
\begin{align}\label{eq:CGR}
  f(x) =   p(i) \times \prod_{j=1}^{r}f(y_{j} | x_{\pas{j}}),
\end{align}
with $\pas{j} := \pa{y_{j}}$. The univariate conditional densities in the product \eqref{eq:CGR} are still Gaussian with a conditional mean depending on the parents. Such models are called CG regressions, see e.g.\ \cite{edwards2012introduction}. One particular useful feature of decomposable models is, that the maximum likelihood estimates of the parameters in \eqref{eq:CG} can be obtained from those in \eqref{eq:CGR}. We now derive the maximised likelihood of \eqref{eq:CGR}, which is then exploited in Section~\ref{sec:od} in order to arrive at a test statistic to be used in connection with outlier detection.

Suppose we have a sample of i.i.d.\ observations, $x^{\ell} = (i^{\ell}, y^{\ell})$ for $\ell=1, 2, \ldots, |n|$, from a decomposable mixed graphical model, and let  $\bm{x} = (x^{1}, x^{2}, \ldots, x^{|n|})$ be the vector of observations. The likelihood of the $j$'th Gaussian factor then takes the form
\begin{align}\label{eq:likelihood_j}
  L(\theta_{j}; \bm{x}) &=  \prod_{\ell = 1}^{|n|} f(y_{j}^{\ell} \mid x_{\pas{j}}^{\ell}) =  \biggl( \prod_{\ell = 1}^{|n|} \sigma^{2}(i_{\pas{j}})^{-1/2} \biggr) \times \exp\biggl\{ -\frac{1}{2}\sum_{\ell = 1}^{|n|}\sigma^{-2}(i_{\pas{j}})(y_{j}^{\ell} - \mu(x_{\pas{j}}^{\ell}))^{2} \biggr\},
\end{align}
where $\mu(x_{\pas{j}})$ and $\sigma^{2}(i_{\pas{j}})$ are the conditional variance and mean of $Y_{j}$ given $X_{\pas{j}} = x_{\pas{j}}$, respectively, and $\theta_{j}$ is the set of parameters. Note, that the variances only depend on the cell values. From here, we simply write $\mu_{j}(x) : = \mu(x_{\pas{j}})$ and $\sigma_{j}(i):=\sigma(i_{\pas{j}})$ to ease notation. The means are assumed to have linear parameterisations of the form
\begin{align*}
  \mu_{j}(x) = \alpha_{j}(i) + \beta_{j}^{T}(i)y_{\pas{j}}, \quad \text{for } i \in \amarg{\pas{j}},
\end{align*}
where $\alpha_{j}$ is a real-valued function of the cells, and $\beta_{j}$ is a real-valued vector function of the cells with dimension $|\pas{j} \cap \Gamma|$. Define the subset of observations in cell $i_{a}$ by $\eta(a) := \{\ell \mid i_{a}^{\ell} = i_{a} \}$. Then, the likelihood in \eqref{eq:likelihood_j} can be written as the product of simple Gaussian likelihoods
\begin{align*}
  L(\theta_{j}; \bm{x}) &= \biggl( \prod_{i \in \amarg{\pas{j}}}\sigma_{j}^{2}(i)^{-n_{\pas{j}}(i)/2} \biggr) \times \exp\biggl\{ -\frac{1}{2}\sum_{i \in \amarg{\pas{j}}} \sigma_{j}^{-2}(i) \sum_{k\in\eta(\pas{j})}(y_{j}^{k} - \mu_{j}(x^{k}))^{2} \biggr\} \\
  &= \prod_{i \in \amarg{\pas{j}}} \biggl ( \sigma_{j}^{2}(i)^{-n_{\pas{j}}(i)/2} \times \exp\biggl\{ -\frac{1}{2\sigma_{j}^{2}(i)}\sum_{k\in\eta(\pas{j})}(y_{j}^{k} - \mu_{j}(x^{k}))^{2}  \biggr\} \biggr),
\end{align*}
implying that the sum of squares depends on a particular cell as $x^{k} = (i^{k}, y^{k})$. The set of parameters can be written as $\theta_{j} = \cup_{i \in \mathcal{I}_{\pas{j}}} \{ \alpha_{j}(i), \beta_{j}(i), \sigma_{j}^{2}(i)\}$. When $\pas{j} \cap \Delta = \emptyset$, we define $\mathcal{I_{\emptyset}}$ as the empty set and $\eta(\emptyset) := \{1,2\ldots, |n|\}$. In this case, the likelihood reduces to the ordinary Gaussian likelihood, where the mean only depends on continuous variables, and the variance is homogeneous. That is,
\begin{align*}
  L(\theta_{j}; \bm{x}) = (\sigma_{j}^{2})^{-|n|/2} \times \exp \biggl\{ -\frac{1}{2\sigma_{j}^{2}}\sum_{\ell = 1}^{|n|} (y_{j}^{\ell} - \mu_{j}(y^{\ell}))^{2} \biggr\},
\end{align*}
where $\theta_{j} := \{ \alpha_{j}, \beta_{j}, \sigma_{j}^{2} \}$. 
For the pure discrete factors, the likelihood is denoted by
\begin{align}\label{eq:discrete_likelihood}
L(p; n) = \prod_{i\in \mathcal{I}}p(i)^{n(i)}, 
\end{align}
where $n$ is the table of counts and $p = \{p(i)\}_{i\in\mathcal{I}}$. Hence, the complete likelihood is given by
\begin{align}\label{eq:complete_likelihood} 
  L(\theta; \bm{x}) = L(p; n) \prod_{j=1}^{r} L(\theta_{j}; \bm{x}) ,
\end{align}
where $\theta = \cup_{j=1}^{r} \{ \theta_{j} \} \cup \{p\}$. Finally, using standard results for linear normal models, the maximum of the likelihood takes the form
\begin{align}\label{eq:max_likelihood}
  L(\hat \theta; \bm{x}) &= \prod_{j = 1}^{r} \biggl( \prod_{i \in \amarg{\pas{j}}}  \hat\sigma_{j}^{2}(i)^{-n_{\pas{j}}(i)/2}  \times \exp \bigl\{-n_{\pas{j}}(i)/2 \bigr\} \biggr) \times \prod_{i \in \amarg{}}\hat p(i)^{n(i)},
\end{align}
where
\begin{align*}
  \hat\sigma_{j}^{2}(i) = \frac{1}{n_{\pas{j}}(i)} \sum_{k\in\eta(\pas{j})}(y_{j}^{k} - \hat \mu_{j}(x^{k}))^{2} \qquad \text{and} \qquad \hat \mu_{j}(x^{k}) = \hat \alpha_{j}(i^{k}) + \hat \beta_{j}^{T}(i^{k})y_{\pas{j}}^{k}, \quad \text{for } i \in \amarg{\pas{j}},
\end{align*}
and the linear parameters are estimated by ordinary least squares. Let $C_{1}, C_{2}, \ldots, C_{K}$ be a sequence of cliques in $G_{\Delta}$ satisfying the RIP ordering. It can then be shown \citep{lauritzen1996graphical} that
\begin{align*}
  \hat p(i) = \frac{1}{|n|} \frac{\prod_{k=1}^{K}n_{C_{k}}(i_{C_{k}})}{\prod_{k=2}^{K}n_{S_{k}}(i_{S_{k}})}, \quad \text{for } i\in \mathcal{I},
\end{align*}
which is the maximum likelihood estimates of \eqref{eq:discrete_likelihood} as also exploited in the outlier detection model given in \cite{lindskououtlier}.

\section{The Null Hypothesis and Deviance Test Statistic}
\label{sec:od}
We aim to test if the observation $z = (i^{0}, y^{0}) := x^{|n|}$ is
an outlier, i.e.\ it deviates significantly from all other
observations. Suppose that $i^{0}$ is an observation sampled from a
distribution, $q$, different from $p$, the distribution of
$i^{1}, i^{2}, \ldots, i^{|n|-1}$ and that
$Y_{j}^{0} \mid Z_{\pas{j}} = (i^{0}_{\pas{j}}, y^{0}_{\pas{j}}) \sim
N(\lambda_{j}^{0}, \sigma_{j}^{2}(i^{0}))$. Then, the null hypothesis
takes the compound form
\begin{align*} 
H_{0}: \bigl\{\lambda_{j}^{0} = \alpha_{j}(i^{0}) + \beta_{j}(i^{0})^{T} y_{\pas{j}}^{0}, \ j = 1, 2, \ldots, r\bigr\} \quad \wedge \quad \bigl\{q = p\bigr\}.
\end{align*}

Let
$E_{0,j}:= \bigl\{\alpha_{j}(i), \beta_{j}(i);
i\in\mathcal{I}_{\pas{j}}\bigr\}$, and define the set of mean parameters
under $H_{0}$ as $E_{0} = \cup_{j=1}^{r}E_{0,j}$. The set of mean
parameters under the alternative hypothesis is then given by
$E = \cup_{j=1}^{r} E_{j}$, where
$E_{j} = E_{0,j} \cup \{\lambda_j\}$ and $\lambda_j$ is a
single parameter describing the conditional mean of $Y^{0}_{j}$ given
$Z_{\pas{j}} = (i^{0}_{\pas{j}}, y^{0}_{\pas{j}})$. Hence, under the alternative hypothesis, the mean $\lambda_j$ of $Y_{j}^{0}$ is not restricted and is free to vary, and hence
$\hat \lambda_j = y_{j}^{0}$. Let
\[
  \theta_{0} = E_{0} \cup \{p\} \cup \{\sigma_{j}^{2}(i)\}_{j=1,2,\ldots,r, \ i\in\mathcal{I}_{\pas{j}}} \quad \text{and} \quad \theta = E \cup \{q\} \cup \{\sigma_{j}^{2}(i)\}_{j=1,2,\ldots,r, \ i\in\mathcal{I}_{\pas{j}}}.
\]
The likelihood ratio is then given by
\begin{equation*}
  LR(z) = \frac{L(\hat \theta_{0}; \bm{x})}{L(\hat \theta; \bm{x})}.
\end{equation*}
Using \eqref{eq:max_likelihood}, we obtain
\begin{equation*}
  LR(z) = \prod_{j = 1}^{r} \prod_{i \in \amarg{\pas{j}}} \biggl(\frac{\hat\sigma_{j}^{2}(i)}{\hat \sigma_{j,0}^{2}(i)} \biggr)^{n_{\pas{j}}(i)/2} \times \frac{L(\hat p;n)}{L(\hat q;n)}
  = \prod_{j = 1}^{r} \biggl(\frac{\hat\sigma_{j}^{2}(i^{0})}{\hat \sigma_{j,0}^{2}(i^{0})} \biggr)^{n_{\pas{j}}(i_{\pas{j}}^{0})/2} \times \frac{L(\hat p;n)}{L(\hat q;n)},
\end{equation*}
by exploiting that the two variance estimates coincide in all cells
but $i_{\pas{j}}^{0}$. Further define
\begin{align}\label{eq:Qj}
 Q_{j} := \hat\sigma_{j}^{2}(i^{0}) / \hat \sigma_{j,0}^{2}(i^{0}),
\end{align}

and let the degrees of freedom $\df_{j} = n_{\pas{j}}(i^{0}_{\pas{j}}) - |\pas{j} \cap \Gamma| - 1$. Using that $E_{0,j}$ and $E_{j}$ differ by exactly one parameter, $\lambda_{j}$, together with Cochran's theorem \citep{cochran1934distribution}, it then follows that
\begin{align}\label{eq:Qj_dist}
  Q_{j} \sim \text{Beta}(\df_{j}/2, 1/2), \quad \text{for } n_{\pas{j}}(i^{0}_{\pas{j}}) > |\pas{j} \cap \Gamma| + 1. 
\end{align}
The likelihood ratio for the pure discrete part,
$Q_{D} := L(\hat p;n) / L(\hat q;n)$, was investigated by
\cite{lindskououtlier}: Given a RIP ordering
$C_{1}, C_{2}, \ldots, C_{K}$ of the cliques in $G_{\Delta}$, it was shown
that
\begin{align}\label{eq:disc_likelihood}
  -2\log Q_{D}(z) = -2\biggl ( & \sum_{k=1}^{K}H(n_{C_k}(i_{C_k}^{0})) -\sum_{k=2}^{K}H(n_{S_k}(i_{S_k}^{0})) - H(|n|) \biggr),
\end{align}
where $H(x) := G(x-1) - G(x)$ and
\begin{align*}
  G(x) = \begin{cases}
    0        & \mbox{ if } x \leq 0 \\
    x\log(x) & \mbox{ if } x > 0.
    \end{cases}
\end{align*}

The total deviance test statistic for testing $H_{0}$ is therefore given by
\begin{align}\label{eq:dev}
 D(z) := -2\log LR(z) = -\sum_{j = 1}^{r}n_{\pas{j}}(i^{0}_{\pas{j}})\log Q_{j} -2\log Q_{D}(z).
\end{align}

The following result state, that the quantities in \eqref{eq:Qj} can be calculated independently when $i^{0}$ is known.

\begin{proposition}\label{prop:qj_indep}
  The quantities $Q_{1}, Q_{2}, \ldots, Q_{r}$ are jointly independent given $i^{0}$.
\end{proposition}

\begin{proof}
 Notice first, that the distribution of $Q_{j}$ in \eqref{eq:Qj_dist} is conditional on $x^{0}_{1}, \ldots, x^{0}_{j-1}$. But since the distribution only depends on $\df_{j}$, we conclude that $Q_{j}$ is independent of $Q_{1}, Q_{2}, \ldots, Q_{j-1}$ given $i^{0}$. Repeating this argument for $j = r, r-1, \ldots, 2$, the result follows.
\end{proof}

\subsection{A Note on Studentized Residuals}
\label{sec:stud_res}
For each $j$, the ratio $Q_{j}$ can be used for an outlier test on data that conforms with $i^{0}_{\pas{j}}$ using studentized residuals. Recall, that a studentized residual is of the form
\begin{align*}\label{eq:r}
  r_{j}^{0} = \frac{y_{j}^{0} - \tilde \mu_{j}(z)}{\sqrt{\tilde \sigma_{j}^{2}(i^{0})(1 - h_{j}^{0})}} \sim t_{\df_{j}},
\end{align*}
where $h_{j}^{0}$ is the so called leverage, which is the
$j$'th diagonal element of the hat matrix, and where
$\tilde \sigma^{2}$ and $\tilde \mu$ are the estimates under the
alternative hypothesis, i.e.\ excluding $z$. Let $f =
(r_{j}^{0})^{2}$. Then, $f \sim F_{1, \df_{j}}$, and since the
Beta distribution is mirror-symmetric, we obtain
\begin{align*}
 Q_{j} := 1 - \frac{f/\df_{j}}{1 + f/\df_{j}} \sim \text{Beta}(\df_{j}/2, 1/2).
\end{align*}
Hence, the contribution of the $j$'th ratio $Q_j$ in \eqref{eq:dev} is large when $|r_{j}^{0}|$ is large, i.e.\ when $y_{j}^{0}$ is deviating from the expectation under $H_{0}$ in cell $i_{\pas{j}}^{0}$.

\subsection{The Homogeneous Case}
\label{sec:hom}
In the homogeneous case, the conditional variance of $Y_{j}$ given $X_{\pas{j}} = (i_{\pas{j}}, y_{\pas{j}})$ is assumed to be independent of the discrete parents. That is, it is assumed that $\sigma_{j}^{2}(i) = \sigma^{2}_{j}$ for all $i\in \mathcal{I}_{\pas{j}}$. It follows that the maximised likelihood function in \eqref{eq:max_likelihood} reduce to
\begin{align*}
\prod_{j = 1}^{r}  (\hat \sigma_{j}^{2})^{-|n|/2} \times \exp\{ -|n|/2\} \times \prod_{i \in \mathcal{I}}\hat p(i)^{n(i)},
\end{align*}
where 
\begin{equation}\label{eq:mu_hom}
  \hat \sigma_{j}^{2} = \frac{1}{|n|} \sum_{i \in \mathcal{I}_{\pas{j}}} \sum_{k\in\eta(\pas{j})}\bigl(y_{j}^{k} - \hat \mu_{j}(x^{k})\bigr)^{2} \qquad \text{and} \qquad \hat \mu_{j}(x^{k}) = \hat \alpha_{j}(i^{k}) + \hat \beta_{j}^{T}y_{\pas{j}}^{k}.
\end{equation}
Notice, that $\hat \beta_{j}$ does not depend on any cells, since otherwise the marginal variance would not be independent of the discrete parents. Under the null hypothesis, the likelihood ratio now takes the form
\begin{align*}
  \prod_{j=1}^{r}\biggl( \frac{\hat \sigma_{j}^{2}}{\hat \sigma_{j,0}^{2}} \biggr)^{-|n|/2} \times Q_{D}.
\end{align*}
Let $Q^{h}_{j} = \hat \sigma_{j}^{2} / \hat \sigma_{j,0}^{2}$, and define $\dfh_{j} = |n| - |\pas{j} \cap \Gamma| - |\mathcal{I}^{+}_{\pas{j}}|$, where $\mathcal{I}^{+}_{\pas{j}}$ is the non-zero cells in $\mathcal{I}_{\pas{j}}$, $j=1,\dots,r$. Then
\[
Q^{h}_{j} \sim \text{Beta}(\dfh_{j}/2, 1/2),
\]
when $\dfh_{j} > 0$. In the case where $\pas{j} \cap \Delta = \emptyset$, the degrees of freedom coincide in the homogeneous and inhomogeneous case, such that $\dfh_{j} = \df_{j} = |n| - |\pas{j} \cap \Gamma| - 1$.

\subsection{Evaluating Deviances}
\label{sec:est_pars}
In order to evaluate the deviance, $D(z)$, for a new observation, $z$, in the inhomogeneous case, one must compute the variance estimates $\hat \sigma_{j,0}^{2}(i^{0})$ and $\hat \sigma_{j}^{2}(i^{0})$ to obtain $Q_{j}$ for $j = 1, 2, \ldots, r$. Similarly, in the homogeneous case, the estimates $\hat \sigma_{j}^{2}$ and $\hat \sigma_{j,0}^{2}$ must be computed to obtain $Q_{j}^{h}$ for $j = 1, 2, \ldots, r$. In the following, we give efficient methods for the calculations.

\subsubsection{Inhomogeneous Case}
\label{sec:inhom}
A natural way of estimating the variances in the inhomogeneous case is by fitting two linear regression models, one under the null hypothesis and one under the alternative hypothesis. Exploiting the connection to studentized residuals, it is only required to fit a single linear regression model under the alternative hypothesis (i.e.\ excluding $z$) and then calculate the quantities $Q_{j} = 1 - (f/\df_{j})/(1+f/\df_{j})$ as explained in Section~\ref{sec:stud_res}.

\subsubsection{Homogeneous Case}
\label{sec:inhom}
In order to estimate the variances in the homogeneous case using linear regression, $(|\pas{j} \cap \Gamma| + |\mathcal{I}^{+}_{\pas{j}}|)\times (|\pas{j} \cap \Gamma| + |\mathcal{I}^{+}_{\pas{j}}|)-$dimensional matrices must be inverted. Such inversions can be expensive even when $|\pas{j}|$ is small since $|\mathcal{I}^{+}_{\pas{j}}|$ may be large if some of the discrete variables have many levels. However, it is not of interest to know the estimated mean parameters; these are only required to estimate the variances and hence calculate $Q^{h}_{j}$. We circumvent this problem by centring the observations. As a consequence, we only need to invert matrices of dimension $|\pas{j} \cap \Gamma| \times |\pas{j} \cap \Gamma|$. See Appendix~\ref{sec:var_est} for details.

\section{The Outlier Test}
\label{sec:outlier_test}

In this section, we summarise the results of the previous sections and
suggest a novel outlier detection procedure, ODMGM, using CGR models in
Algorithm~\ref{alg:odmgm}. We first reiterate the
method given in \cite{lindskououtlier} for simulating discrete cells
in Algorithm~\ref{alg:sim_cells}, which is needed in
Algorithm~\ref{alg:odmgm}. The method is based on a RIP ordering of
the cliques in a pure discrete graph and, exploiting, the chain rule
\[
  P(I = i) = P(I_{C_{1}} = i_{C_{1}}) \prod_{k=2}^{K}P(I_{C_{k}} = i_{C_{k}} \mid I_{S_{k}} = i_{S_{k}}),
\]
where the RIP ordering ensures that the cell value $i_{S_{k}}$ is
known, since it holds that $S_{k} \subset C_{j}$ for some $j < k$
(this is exploited in line~\ref{line:sk} of
Algorithm~\ref{alg:sim_cells}).

\begin{algorithm}[h]
\caption{Simulate Cells in Pure Discrete Decomposable Graphical Models}\label{alg:sim_cells}
\begin{algorithmic}[1]
  \Procedure{}{$G$: Pure discrete decomposable graph. $U$: Dataset of  observations}
  \State Form the contingency table $n$ of all observations in $U$
  \State Construct a sequence of cliques, $C_{1}, C_{2}, \ldots, C_{K}$, having RIP from $G$
  \State Let $i := \{\}$ be an ordered list
  \State Simulate $i_{C_{1}}$ using the probability table $n_{C_{1}}(i_{C_{1}}) / |n|$ and append $i_{C_{1}}$ to $i$
    \For{$k = 2, 3, \ldots, K$}
    \State Simulate $i_{C_{k}\setminus S_{k}}$ using the conditional probability table $n_{C_{k}}(i_{C_{k}\setminus S_{k}}, i_{S_{k}}) / n_{S_{k}}(i_{S_{k}})$ \label{line:sk}
    \State Append $i_{C_{k}\setminus S_{k}}$ to $i$
  \EndFor 
 \EndProcedure
\end{algorithmic}
\end{algorithm}

\begin{algorithm}[h]
\caption{Outlier Detection in Mixed Graphical Models (ODMGM)}\label{alg:odmgm}
\begin{algorithmic}[1]
  \Procedure{}{$G$: Decomposable mixed graph, $U$: Dataset of observations, $z$: New observation.}
  \State Append $z$ to $U$ and form the contingency table $n$ of all observations
  \State Find a perfect ordering of the vertices in $G$
  \State Construct the pure graph $G_{\Delta}$
  \For{$\ell = 1,2,\ldots, N$}
    \State Simulate cell $i^{\ell}$ by applying Algorithm~\ref{alg:sim_cells} on $G_{\Delta}$
      \For{$j = 1,2, \ldots, r$}
         \If{$\df_{j}\geq 0$}
         \State simulate $Q_{j}$ from $\text{Beta}(\df_{j}/2, 1/2)$  \label{line:qj}
         \EndIf
      \EndFor 
    \State Calculate $D(x^{\ell})$ by applying \eqref{eq:disc_likelihood} to cell $i^{\ell}$ and add it to $-\sum_{j:\df_{j} \geq 0} n_{\pas{j}}(i^{\ell}_{\pas{j}})\log(Q_{j})$ \label{line:dx}
    \EndFor
    \State Define the empirical CDF, $F(x) = N^{-1}\sum_{\ell = 1}^{N}\mathbbm{1}[D(x^{\ell})\leq D(x)]$, and calculate $D(z)$ using \eqref{eq:dev}
    \If{$F(z) \geq 1 - \alpha$}
    \State declare $z$ as outlier in $U$ at an $\alpha-$level
    \EndIf
 \EndProcedure
\end{algorithmic}
\end{algorithm}

Notice that, in Algorithm~\ref{alg:odmgm} the new observation, $z$, is
appended to the data, $U$; i.e.\ under the null
hypothesis it is assumed that $z$ originates from the same generating
process as all the observations in $U$.  Next, due to the results in
\eqref{eq:Qj_dist}, it is not necessary to simulate the associated
continuous part of each simulated cell in order to simulate the
deviances. This implies a large reduction in the computational time
needed for simulation. As a consequence of Proposition~\ref{prop:qj_indep}, the quantities $Q_{j}$, can be computed in parallel due to conditional independence.

The homogeneous version follows by replacing $\df_{j}$ and $Q_{j}$
with their respective counterparts, $\dfh_{j}$ and $Q^{h}_{j}$, and
replacing all $n_{\pas{j}}(i_{\pas{j}})$ in line~\ref{line:dx} with
$|n|$.

\section{Real Data Example}
\label{sec:real}
In this section, we apply ODMGM to the cover type (CT) data from the
UCI Machine Learning Repository \citep{dua2020}.  This dataset demands the usage of non-trivial models in order to capture the large
amount of information. This has caught the attention of
researchers in the machine learning community in order to benchmark
different classification models. Each sample in the data is taken from
a $30m \times 30m$ patch of forest that is classified as one of seven
CTs represented as integers: $1$: Spruce/Fir ($37\%$), $2$: Lodgepole
Pine ($48\%$), $3$: Ponderosa Pine ($6\%$), $4$: Cottonwood/Willow
($1\%$), $5$: Aspen ($2\%$), $6$: Douglas-fir ($3\%$) and $7$: Krummholz
($4\%$). In addition, the CT data contains $53$ explanatory variables
of which $44$ are discrete with two levels (i.e.\ binary). Of these,
$40$ describe the presence (or absence) of a particular soil type, and
four describes the presence (or absence) of the wilderness area. The
remaining variables are continuous and includes for example elevation, slope, horizontal distance to hydrology and
hillshade at noon. \citet{dua2020} gave a thorough explanation of the
entire dataset. The original dataset consists of $581,000$ samples,
however, we have down-sampled to $20,000$ samples to keep the CPU running
time down while preserving the frequency distribution of the CTs.

Recently, \cite{kumar2020classification} applied a random forest model
to obtain a classification accuracy of 95\% when predicting the CT of
a sample. We demonstrate, that classification should be conducted with
caution since, in many situations, more than a single CT is a
statistically plausible explanation of a sample. Furthermore, some
authors, e.g.\ \cite{zhiwei2017research}, assumed the explanatory
variables to be independent, which we show is an invalid exorbitant
assumption.

We start the analysis by fitting an interaction graph using the
\texttt{R} package \texttt{gRapHD} \citep{graphhd_edwards} to investigate the complexity of the CT data. The interaction graph, $G$, is
depicted in Figure~\ref{fig:interaction_graphs} (left), where white
vertices represent continuous variables, black vertices represent
discrete variables, and the grey vertex is the class variable. Clearly,
the explanatory variables are associated with CT either by a direct
relation or implicitly through other explanatory variables. The
interaction graph is rather complex and it is thus questionable to assume independence among all variables. There are four isolated variables, i.e.\ they are not connected to any other variable in the graph. We have removed these four variables (columns $21$, $22$, $50$ and $51$ in the UCI dataset)
to reduce the complexity.

In order to benchmark ODMGM as an outlier tool, we construct seven
interaction graphs, one for each class. Figure
\ref{fig:interaction_graphs} shows one of the more complex
interaction graphs for class $1$ ($G_1$, middle) and the simplest one for
class 7 ($G_{7}$, right). Notice, that these are quite different
implying that samples from different classes are, most likely, generated
by different mechanisms with intrinsic association differences among
the explanatory variables.

In Section~\ref{sec:pm}, we use these interaction graphs to calculate
the proportion of samples that ODMGM is able to declare as outliers
in each class. We benchmark the results to those of iForest. In the
following section, we investigate if the underlying assumptions of the
CGR model in class $7$, which was chosen for simplicity
due to its simpler interaction structure compared to the other
classes, is valid.

\begin{figure}[h!] 
  \centering
  \begin{minipage}[h]{.32\linewidth}
    \centering
    {\includegraphics[bb={90 80 440 430},clip=true,scale=0.30]{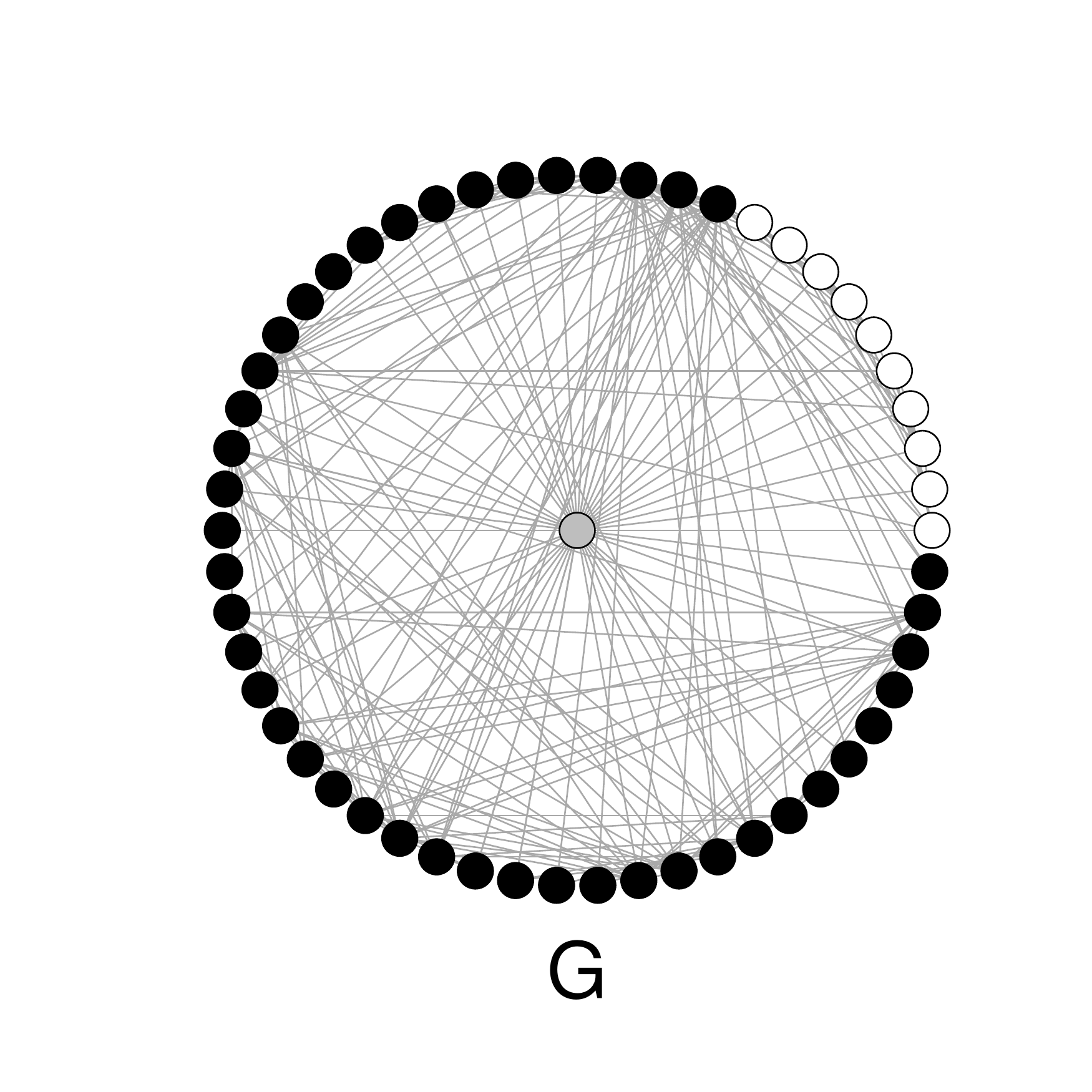}}\\
    $G$
  \end{minipage}
  \begin{minipage}[h]{.32\linewidth}
    \centering
    \includegraphics[bb={90 80 440 430},clip=true,scale=0.30]{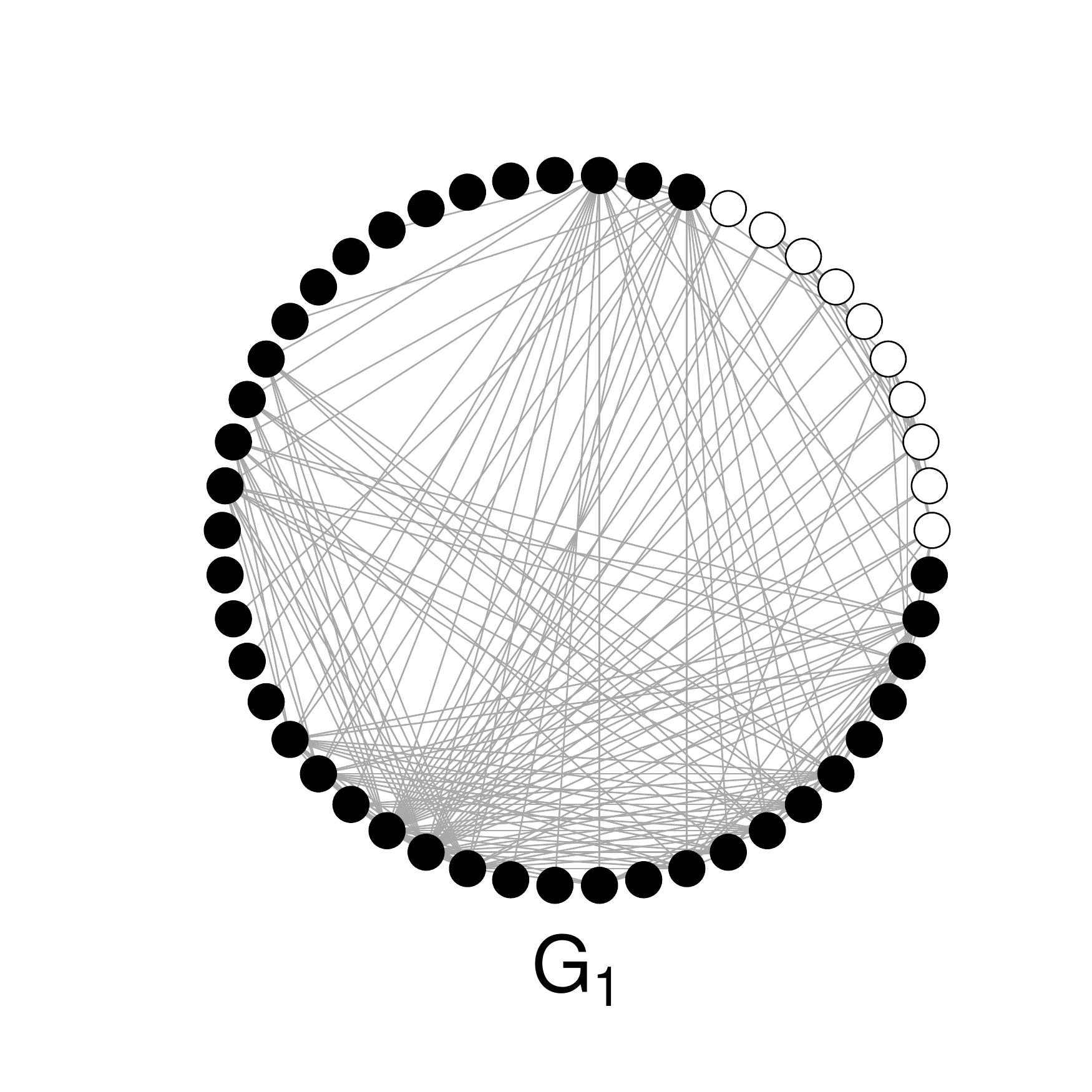}\\$G_1$ 
  \end{minipage}
  \begin{minipage}[h]{.32\linewidth}
    \centering
    \includegraphics[bb={90 80 440 430},clip=true,scale=0.30]{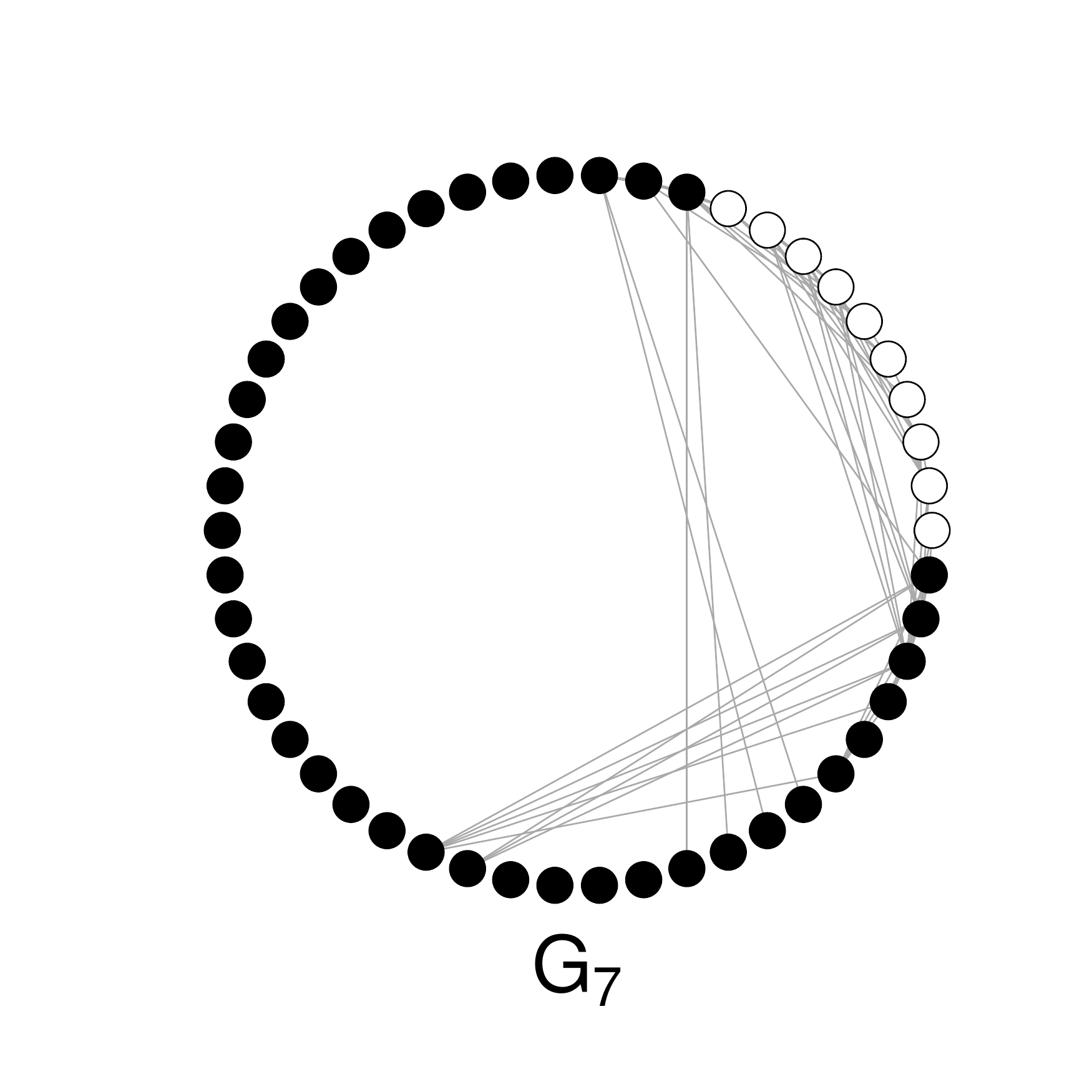}\\$G_7$
  \end{minipage}
    \caption{\label{fig:interaction_graphs}Left: Interaction graph, $G$,
    for the down-sampled CT data including the class variable. Middle:
    Interaction graph, $G_{1}$, for class 1. Right: Interaction graph,
    $G_{7}$, for class 7. Discrete variables are black,
    continuous variables are white and the central, grey
    vertex in the left graph represent the class variable.}
\end{figure}

\subsection{Verifying CGR Assumptions}
\label{sec:cgr_assumptions}
Consider the subgraph $H_{7}$ of $G_{7}$ in Figure~\ref{fig:subgraph7},
where we have named the vertices according the column position of the
corresponding variables in the cover type data.

\begin{figure}[h!]
  \centering
  \includegraphics[bb={90 80 460 450},clip=true,scale=0.4]{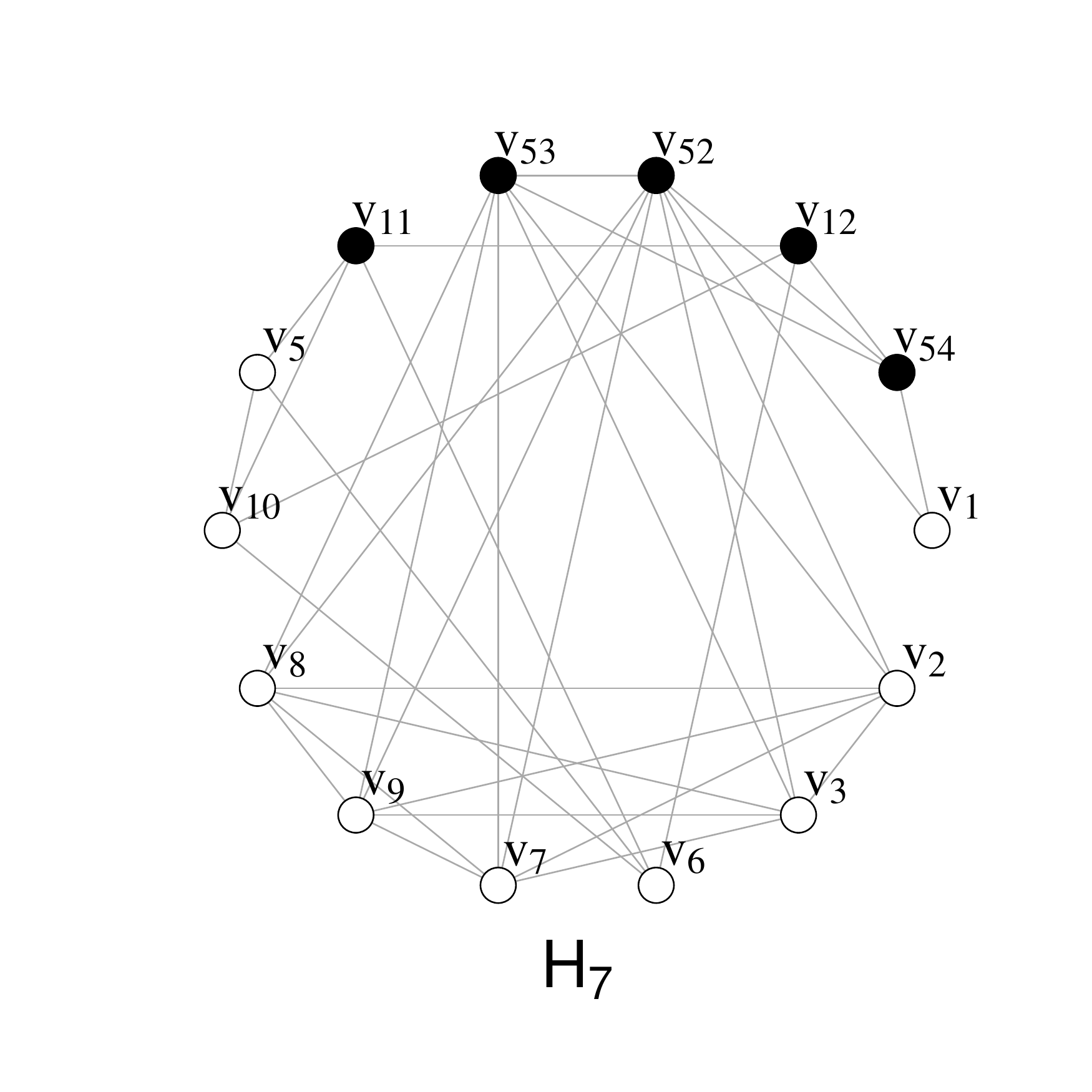}\\$H_7$
  \caption{\label{fig:subgraph7} The subgraph $H_{7}$ of $G_{7}$
    illustrating some of the associations between the CT variables for
    class $7$.}
\end{figure}

The subgraph $H_7$ in Figure~\ref{fig:subgraph7} consists of the
continuous variables and the corresponding discrete parents. Using MCS
together with Proposition~\ref{prop:cont_before_disc}, a perfect
numbering of the vertices can be computed such that
$\pa{v_1} = \{v_{52}, v_{54}\}$, $\pa{v_5} = \{v_6,v_{10}, v_{11}\}$,
$\pa{v_7} = \{v_2, v_3, v_{52}, v_{53}\}$,
$\pa{v_8} = \{v_2, v_3, v_7, v_{52}, v_{53}\}$,
$\pa{v_9} = \{v_2, v_3, v_7, v_8, v_{52}, v_{53}\}$ and
$\pa{v_{10}} = \{v_{11}, v_{12}\}$. First, we make a graphical check
for $v_1$ and $v_{10}$ being approximately Gaussian given their
parents, see Figure~\ref{fig:gaussian_given_disc}. In the light of a
rather complex model, the density plots in Figure
\ref{fig:gaussian_given_disc} look fairly symmetric and bell-shaped
for $v_{1}$ (top row, Figure~\ref{fig:gaussian_given_disc}). For
$v_{10}$, the densities are neither symmetric nor bell-shaped, but the
deviations from the Gaussian distribution are not large (bottom
row, Figure~\ref{fig:gaussian_given_disc}).

\begin{figure}[h!]
  \centering
  \includegraphics[scale=0.5]{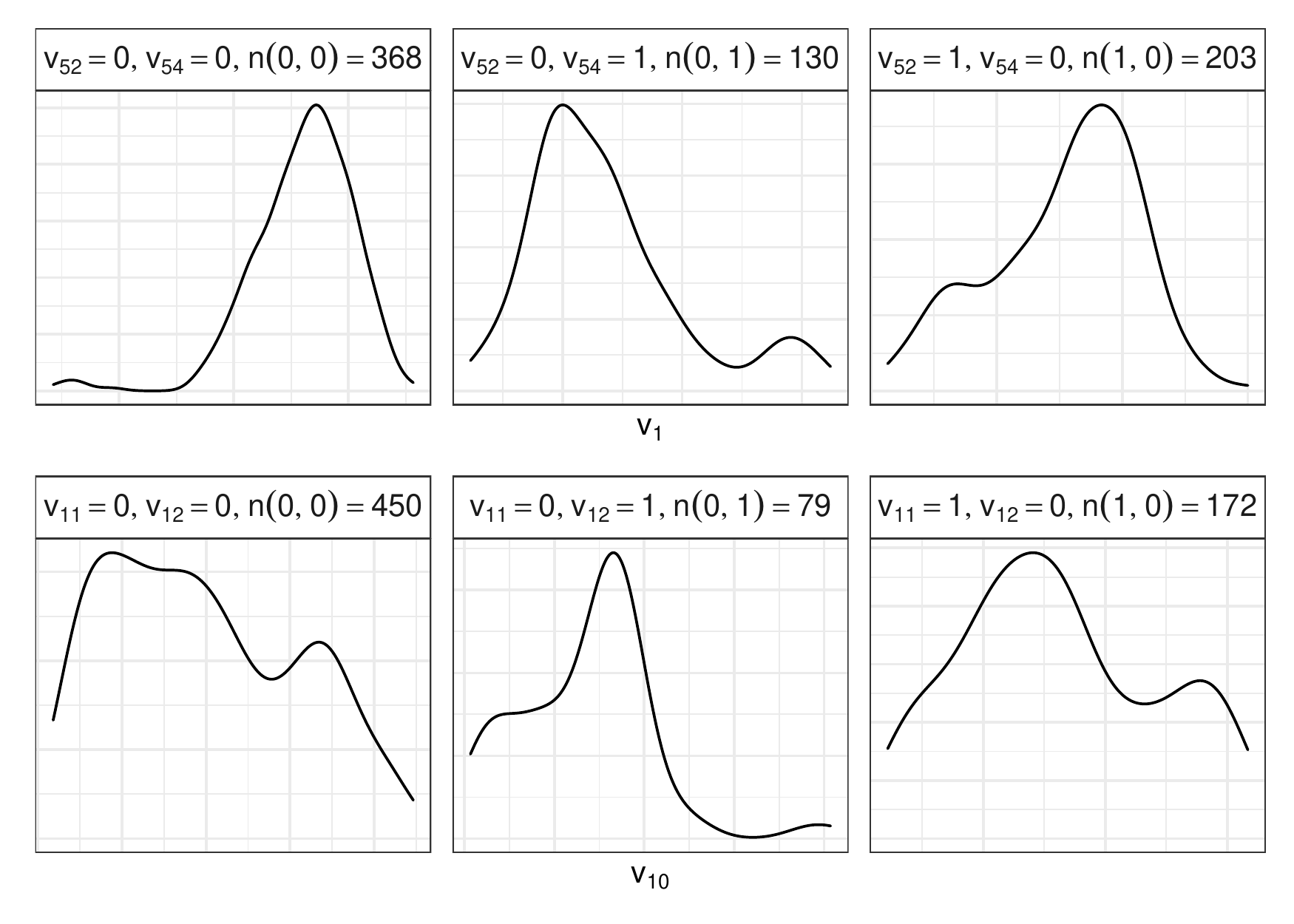}
  \caption{\label{fig:gaussian_given_disc}Density plots of $v_{1}$
    (top row) and $v_{10}$ (bottom row) given the configurations of
    their discrete parents. The panel headers indicate the level of
    the discrete parents, e.g.\ $v_{52} = 0, v_{54}=0$, and the numbers of observations, e.g.\ $n(0,0) = 368$.}
\end{figure}

It is difficult graphically to verify, whether $v_5, v_7, v_8$ and
$v_9$ are Gaussian with mean values depending on their parents, since
they all have more than one continuous parent. Instead, we shall assess
the adequacy of the assumptions simply by calculating the squared
coefficient of determination, $R^{2}$, for each combination of the
discrete parents. The results are summarised in Table~\ref{fig:r2},
where the numbers represent the values of $R^{2}$ for the given
configurations of $v_{52}$ and $v_{54}$. It can be noticed, that no
samples had the configuration $(v_{52}, v_{54}) = (1,1)$. The values
of $R^{2}$ are overall satisfactory. Notice, that the values for the
model of $v_{9}$ indicates a nearly perfect linear association. The
model of $v_{5}$ has $R^{2}-$values of $0.011$ and $0.258$ for
$v_{11} = 0$ and $v_{11} = 1$, respectively which is less impressive.

\begin{table}[h!]
  \centering
  \caption{Summary of model performance using the squared coefficient
    of determination, $R^2$.}
  \label{fig:r2}  
  \begin{tabular}{lllll} 
    Model      & $(v_{52}, v_{54}):$ & $(0,0)$ & $(1,0)$ & $(0,1)$ \\ \midrule
    $v_{7} \sim v_{2} + v_{3}$                 & & $0.174$  & $0.580$  & $0.396$  \\
    $v_{8} \sim v_{2} + v_{3} + v_{7}$         & & $0.575$  & $0.405$  & $0.599$  \\
    $v_{9} \sim v_{2} + v_{3} + v_{7} + v_{8}$ & & $0.992$  & $0.998$  & $0.988$  \\ \bottomrule
  \end{tabular}
\end{table}

\subsection{Performance}
\label{sec:pm}

We now apply both ODMGM and iForest to the CT data and summarise the
results in Figure~\ref{fig:odmgm_iforest}. Given a specific class, one of the facets, we calculate the
proportion of observations for all other classes that ODMGM and
iForest, respectively, are able to declare as outliers in that specific class. Proportions for ODMGM are shown by circles whereas
results for iForest are shown as filled dots. The grey bands
highlight the proportion of in-class outlier detection; proportions
in this band should optimally equal the significance level, here,
$0.05$. Notably, it is difficult to detect outliers in class $1$ and
$2$ regardless of which method is used. In general, though, ODMGM
outperforms iForest and iForest is in fact worse than random guessing
in the majority of the tests (many rejection fractions less than
$0.5$). Specifically, for class $7$ the difference in performance is
heavily pronounced.

\begin{figure}[h!]
  \centering
  \includegraphics[width=\textwidth]{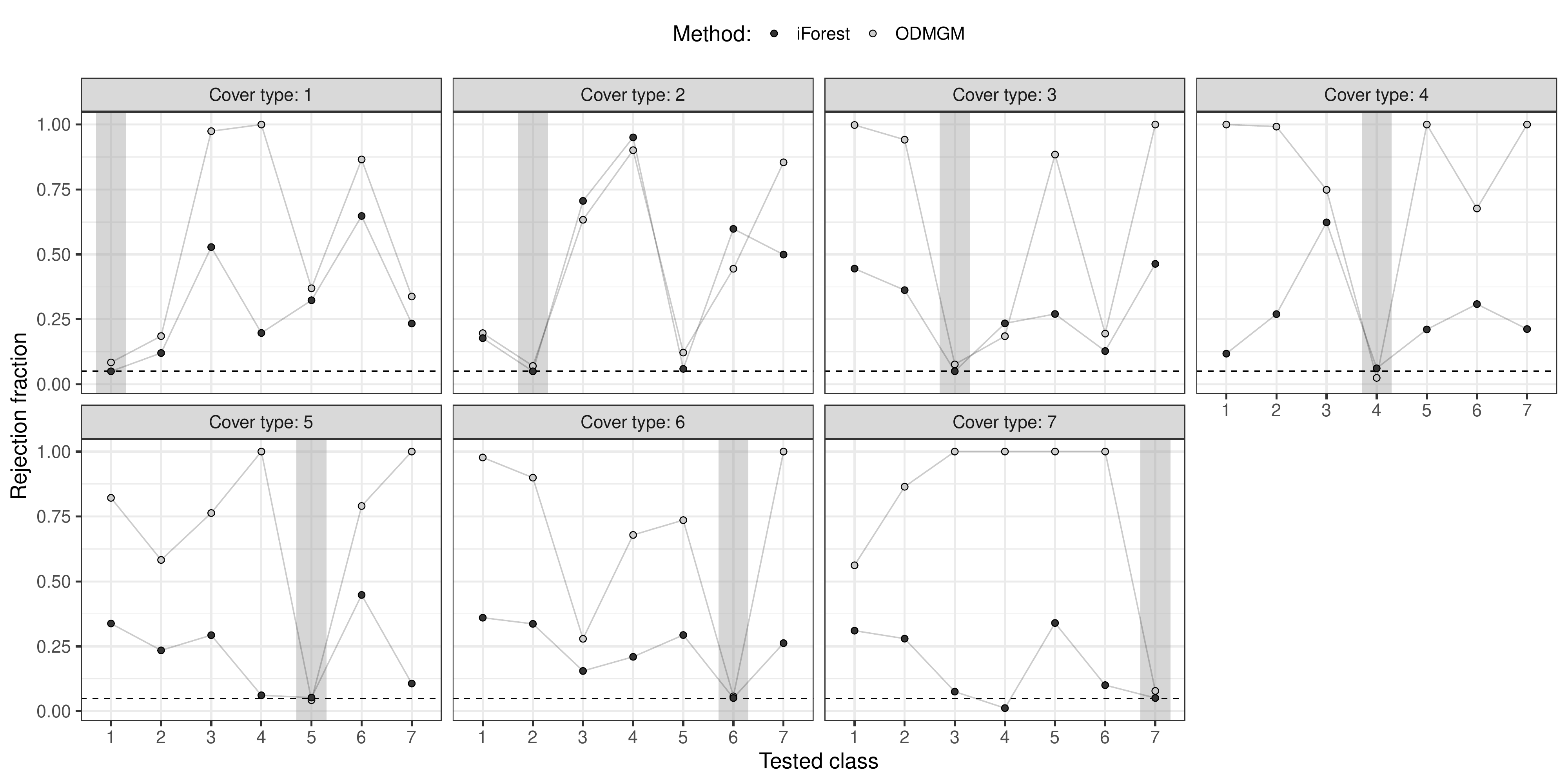} 
  \caption{\label{fig:odmgm_iforest}Each observation in the
    down-sampled dataset was tested as outlier in each class using both
    iForest (dark points) and ODMGM (light points), respectively. The
    panels show how often each of the cover types is rejected in the classes. The grey bands highlights the proportion of
    in-class outlier detection.}
\end{figure}

The presented methodology allows each observation to be tested as
outlier in each of the classes in a dataset. Consequently, an
observation can be declared an outlier in no, some, or all classes. In
the cover type dataset with seven classes, this implied that there may
be up to $2^7 = 128$ different rejection/acceptance combinations. In
Figure~\ref{fig:upsetr}, we plotted the different types of
combinations seen in the down-sampled dataset \citep[created by the
\texttt{UpSetR} \texttt{R}-package,][]{upsetr}. There are 819
observations that were rejected in all classes and $1,784$ observations
accepted in a single class. Furthermore, as expected from
Figure~\ref{fig:odmgm_iforest}, many observations are simultaneously
accepted in both class $1$ and $2$ (e.g.\ $5,479$ in just those two, 
and $6,931$ with an additional class (class $5$: $3,639$, class $7$:
$3,151$, and class $6$: $141$, respectively). The aggregation in
Figure~\ref{fig:upsetr} does not take the true class into
consideration as in Figure~\ref{fig:odmgm_iforest}. However, since the
majority of the observations are of class 1 or 2 (85\%), these also belong to
the classes in which most observations are accepted. Furthermore 77\% ($n = 15,365$) and 80\% ($n = 16,078$) were
accepted as being of class 1 and 2, respectively. This emphasises the
risk of assigning an observation to a single class, which is the
typical approach in a classification setup.

\begin{figure}[hb!]
  \centering
  \includegraphics[width=\textwidth]{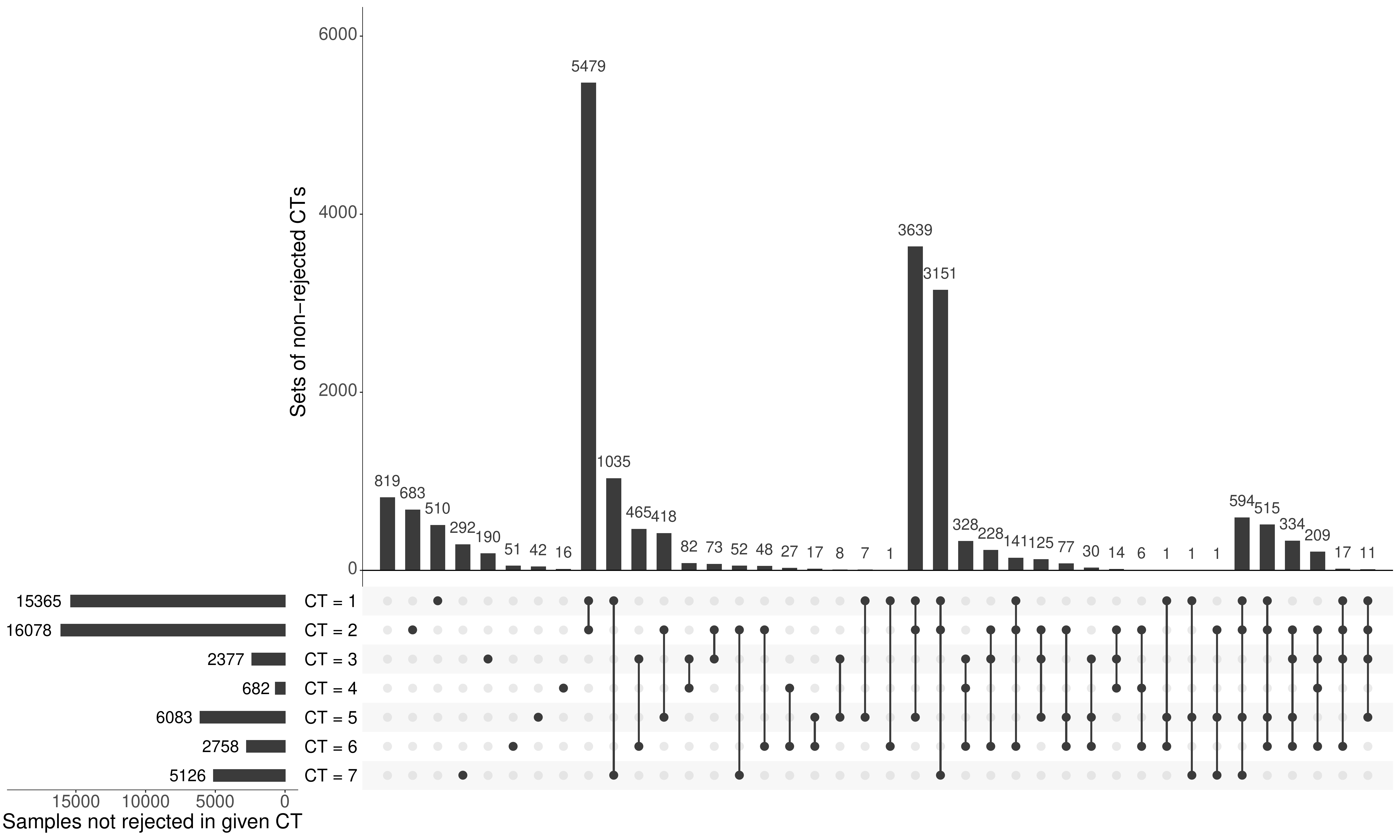}
  \caption{\label{fig:upsetr}Diagram showing the distribution of the various
    combinations of accepted classes for the down-sampled CT
    dataset. Dots (and lines) in the lower part correspond to the
    accepted CT classes \citep[created by the \texttt{UpSetR}
    \texttt{R}-package,][]{upsetr}.}
\end{figure}

\section{Conclusions and Future Work}
In this paper, we present a new probabilistic method, ODMGM, for
outlier detection in high-dimensional data with mixed variables. The
methodology uses a theoretically sound formulation of what is meant by
an outlier. We studied the performance of ODMGM on a real data set and
benchmarked it with the performance of the state-of-the-art algorithm
iForest. We found that ODMGM was superior to iForest and that iForest,
in general, is worse than random guessing for outlier detection for
this particular dataset. This contradicts the findings of several
other authors, see e.g.\ \cite{eiras2019large, xu2019mix,
  aryal2016revisiting, aryal2019improved, domingues2018comparative,
  emmott2015meta}. Furthermore, we saw from Figure~\ref{fig:upsetr}
that in many cases, it is not, statistically, possible to assign a
given sample to a single class. This is in contrast to classification
methods where, if a sample is plausible to originate from two or more
different classes, the method assign the sample to the most probable
class, which may arguably be undesirable in healthcare and forensics
e.g.\ where the cost of a false positive may be fatal. Thus, we
suggest an outlier detection method like ODMGM, and leave the further
investigation to a specialist if it is plausible for a sample to
belong to several classes or to be excluded from all classes.  The
latter case may reveal a new interesting finding. 

Furthermore, we provide software for use in the R language \citep{mlindsk_2020_3999522} together with all code snippets used
to generate all the results in this paper.


Another approach in the homogeneous case, which we hope to investigate
in future research is to assume a different parameterisations of the
conditional mean value in \eqref{eq:mu_hom}. Consider a generic
continuous variable, $y$, and suppose the discrete parents consist of
$i = (i_{1}, \ldots, i_{k})$. The conditional mean of $y$ could then, for example, only include main effects, i.e.
\begin{align*}
  \alpha_{1}(i_{1}) + \alpha_{2}(i_{2}) + \cdots + \alpha_{k}(i_{k}) + \beta^{T}y_{\pa{y}}.
\end{align*}
In this setup, we only require $|n| \geq \sum_j^k |\mathcal{I}_j| - k + |\pa{y} \cap \Gamma| + 1$. All though much simplified, the estimates would be more robust and the model, if appropriate, will have more power. 

Learning a graphical model from high dimensional data is a notoriously hard task, not least because of the many possible structures and the vast amount of data. However, it is more tangible if the graph is assumed to be decomposable since the computational advantages of such an assumption are tremendous. One of the most promising approaches was suggested by \cite{deshpande01_efficient} which offered a detailed algorithm, named ESS, for efficient stepwise model selection in mixed graphical models (including the pure case) is given. \cite{altmuelleril_pract} discovered a flaw in ESS and gave a proof for the correction. For the pure discrete case the ESS algorithm is implemented in the \texttt{R} software package \texttt{ess}, originally a part of the \texttt{molic} package \citep{lindskoumolic}. The ESS algorithm is not yet implemented to handle the mixed case in any known software to our knowledge. The \texttt{R} package \texttt{gRapHD} \citep{edwards10_selec_high_dimen_mixed_graph} was designed for model selection in high-dimensional mixed graphical models. Unfortunately, the package is no longer maintained. To our knowledge, the only maintained \texttt{R} package for model selection in the mixed case is the \texttt{mgm} package. However, one must specify the highest order of interaction in advance and even for small orders the procedure is much too slow for model selection in high-dimensional data. In connection to outlier detection where the procedure may need to run several times, it is crucial that fitting the interaction graph can be done reasonably fast. We plan to implement the ESS procedure for mixed graphs in the future. 

It is well-known, that linear regression models are not robust when data is contaminated with outliers \citep{yu2017robust}. The robustness of the parameter estimators in linear regression is often characterised by the breakdown point which indicates the proportion of outliers that the estimators can resist. It can be shown, that the breakdown point of OLS estimates is $1/|n|$ which tends to zero when the sample size $|n|$ increases. Since the estimates in ODMGM are calculated within sub-tables with $n_{\pas{j}}(i^{0}_{\pas{j}})$ observations, the effective sample size is markedly decreased. Hence, in a way, ODMGM is more robust against outliers compared to a global outlier test that uses all $|n|$ observations for parameter estimation. Practical computations of robust estimates are challenging and therefore increases the computational time. We hope to explore the issue of robustness in connection to ODMGM in more detail in the future to make the method more robust.

\appendix

\section{Variance Estimation for Inhomogeneous Models}
\label{sec:var_est}
First, we need a little more notation, and to ease this, we define $a:=\pas{j}$. Denote by $y_{j}^{k}(i_{a})$ the $k$'th observation of $y_{j}$ in cell $i_{a}$, i.e.\ the observation of $y_{j}$ corresponding to the $k$'th index in $\eta(a)$. Similarly, denote by $y_{a}^{k}(i_{a})$ the $k$'th observation of $y_{a}$ in cell $i_{a}$. The centred observations in cell $i_{a}$ is then given as

\[
  \gamma_{j}^{k}(i_{a}) :=  y_{j}^{k}(i_{a}) - \bar y_{j}(i_{a}), \quad \text{where} \quad \bar y_{j}(i_{a}) = \sum_{k\in\eta(a)} y_{j}^{k}(i_{a}),
\]
and 
\[
  \gamma_{a}^{k}(i_{a}) :=y_{a}^{k}(i_{a}) - \bar y_{a}(i_{a}), \quad \text{where} \quad \bar y_{a}(i_{a}) = \sum_{k\in\eta(a)} y_{a}^{k}(i_{a}).
\]
Our goal is to minimise the sums of squared errors
\[
  \sse_{j}(\alpha, \beta) = \sum_{i_{a}\in \mathcal{I}_{a}} \sum_{k\in\eta(a)}\bigl(y_{j}^{k}(i_{a}) - \alpha(i_{a}) - \beta^{T}y_{a}^{k}(i_{a})\bigr)^{2}.
\]
Using the centred observations, it can be seen that 
\[
  \sse_{j}(\alpha, \beta) = \sum_{i_{a}\in \mathcal{I}_{a}} \sum_{k\in\eta(a)}\bigl(\gamma_{j}^{k}(i_{a})-\beta^{T}\gamma_{a}^{k}(i_{a})\bigr)^{2} + \sum_{i_{a}\in \mathcal{I}_{a}} \sum_{k\in\eta(a)}\bigl(\bar y_{j}(i_{a})- \alpha(i_{a}) - \beta^{T}\bar y_{a}(i_{a})\bigr)^{2}.
\]
Thus,  $\hat \alpha(i_{a}) = \bar y_{j}(i_{a}) - \beta^{T}\bar y_{a}(i_{a})$ and $\hat \beta = S_{a, a}^{-1}S_{a, j}$ where
\[
  S_{u,v} = \sum_{i_{a}\in \mathcal{I}_{a}} \sum_{k\in\eta(a)}\gamma_{u}^{k}(i_{a})\{\gamma_{v}^{k}(i_{a})\}^{T}.
\]
Finally, it follows that
\begin{align}\label{eq:sse}
  \sse_{j}(\hat \alpha, \hat \beta) = S_{j,j} - S_{j,a} S_{a,a}^{-1}S_{a,j} \quad \text {and} \quad \hat \sigma_{j} = |n|^{-1}\sse_{j}(\hat \alpha, \hat \beta).
\end{align}
The problem is now reduced to inverting a $(|\pas{j} \cap \Gamma|) \times (|\pas{j} \cap \Gamma|)$ matrix. The quantities $Q_{j}^{h}$ can thus be computed using \eqref{eq:sse} twice; one with the new observation included and one without. Notice that the minimised sums of squared errors reduces as follows in the special cases:
\begin{align}
  \sse_{j} = \begin{cases}
    S_{j,j} & \quad \text{when} \quad \pas{j} \subset \Delta \\
    \tilde S_{j,j} - \tilde S_{j,a} \tilde S_{a,a}^{-1}\tilde S_{a,j} & \quad \text{when} \quad \pas{j} \subset \Gamma\\
    \tilde S_{j,j} & \quad \text{when} \quad \pas{j} = \emptyset,
  \end{cases}
\end{align}
where $\tilde S_{u,v} = \sum_{\ell =1}^{|n|}\gamma_{u}^{\ell}(\gamma_{v}^{\ell})^{T}$ and $\gamma_{u}^{\ell} = y_{u}^{\ell} - \bar y_{u}$.


\end{document}